\documentclass{amsart}
\usepackage{amsfonts}

\usepackage{amscd,amssymb,amsmath,graphicx,verbatim}
\usepackage[dvips]{hyperref}
\usepackage[TS1,OT1,T1]{fontenc}

\newtheorem{theorem}{Theorem}[section]
\newtheorem{lemma}[theorem]{Lemma}

\newtheorem{proposition}[theorem]{Proposition}

\theoremstyle{definition}

\newtheorem{example}[theorem]{Example}

\theoremstyle{remark}

\numberwithin{equation}{section}

\begin{document}

\title{$2B_{p}$ and $4B_{p}$ are topologically conjugate}

\author{Bingzhe Hou }
\address{Bingzhe Hou, Institute of Mathematics , Jilin University, 130012, Changchun, P.R.China} \email{abellengend@163.com}
\author{Gongfu liao}
\address{Gongfu liao, Department of Mathematics , Jilin Normal University,
136000, Siping, P.R.China} \curraddr{Institute of Mathematics ,
Jilin University, 130012, Changchun, P.R.China}
\email{liaogh@email.jlu.edu.cn}
\author{Yang Cao}
\address{Yang Cao, Institute of Mathematics , Jilin University, 130012, Changchun, P.R.China} \email{caoyang@jlu.edu.cn}

\date{October 22, 2008}
\subjclass[2000]{Primary 37C15; Secondary 47B37}
\keywords{topological conjugacy, weighted backward shift operators,
homeomorphism.}
\thanks{This work was supported by the National Nature
Science Foundation of China (Grant No. 10771084). The first author
is supported by the Youth Foundation of Institute of Mathematics,
Jilin university.}
\begin{abstract}
 Let $\lambda B_{p}$, where $\lambda$ is a nonzero complex number, denote a constant-weighted backward shift
 operators on $l^{p}$ for $1\leq p<\infty$. In this article, we investigate, in
 topologically conjugacy, the complete classification for $\lambda
 B_{p}$.
\end{abstract}
\maketitle

\section{Introduction and preliminaries}

A discrete dynamical system is simply a continuous function $f:
X\rightarrow X$ where $X$ is a complete separable metric space. For
$x\in X$, the orbit of $x$ under $f$ is
$Orb(f,x)=\{x,f(x),f^{2}(x),\ldots\}$ where $f^{n}= f\circ f\circ
\cdots \circ f $ is the $n^{th}$ iterate of $f$ obtained by
composing $f$ with $n$ times. A fundamental but difficult problem is
the classification for dynamical systems in the sense of
"topological conjugacy". If $f: X\rightarrow X$ and $g: Y\rightarrow
Y$ are two continuous mappings, then $f$ is topologically conjugate
to $g$ if there exists a homeomorphism $h: X\rightarrow Y$ such that
$g=h\circ f\circ h^{-1}$, we also say that $h$ conjugates $f$ to
$g$. Notice that "topological conjugacy" is an equivalent relation.
Moreover, it is easy to see that such many properties as, periodic
density, transitivity, mixing $\ldots$, are preserved under
topological conjugacy.

Recall that $f$ is transitive if for any two non-empty open sets
$U,V$ in $X$, there exists an integer $n\geq1$ such that
$f^{n}(U)\cap V\neq \phi$. It is well known that, in a complete
metric space without isolated points, transitivity is equivalent to
the existence of dense orbit (\cite{Silverman}). $f$ is strongly
mixing if for any two non-empty open sets $U,V$ in $X$, there exists
an integer $m\geq1$ such that $f^{n}(U)\cap V\neq \phi$ for every
$n\geq m$. $f$ has sensitive dependence on initial conditions (or
simply $f$ is sensitive)if there is a constant $\delta>0$ such that
for any $x\in X$ and any neighborhood $U$ of $x$, there exists a
point $y\in X$ such that $d(f^{n}(x),f^{n}(y))> \delta$, where $d$
denotes the metric on $X$. Moreover, following Devaney
\cite{Devaney}, $f$ is chaotic if $(a)$ the periodic points for $f$
are dense in $X$, $(b)$ $f$ is transitive, and $(c)$ $f$ has
sensitive dependence on initial conditions. It was shown by Banks
et. al. (\cite{Banks}) that $(a) + (b)$ implies $(c)$ and hence
chaoticity is preserved under topological conjugacy, though
sensitivity is not. For more relative results, we refer to
\cite{Block} and \cite{Devaney}.

We are interested in the dynamical systems induced by continuous
linear operators on Banach spaces. In recent years, there has been
got some improvements at this aspect (Grosse-Erdmann's and Shapiro's
articles \cite{Grosse,Shapiro} are good surveys.).

In operator theory, we have the concept "similarity"; two operators
$A$ and $T$, on Banach spaces $\mathcal {H}$ and $\mathcal {K}$, are
similar if there is a bounded linear transformation $S$ from
$\mathcal {H}$ onto $\mathcal {K}$ with a bounded inverse, such that
$A=S^{-1}TS$. Obviously, it is a stronger relation than topological
conjugacy, i.e., if two operators are similar then they must be
topologically conjugate. Our ultimate aim is to give a complete
classification for continuous linear operators in the topologically
conjugate sense. In the present paper, we restrict our attention to
weighted backward shift operators on $l^{p}$ for $1\leq p<\infty$.
Here $l^{p}$ is the classical Banach space of absolutely $p^{th}$
power summable sequences $x=(x_{1},x_{2},\ldots)$ and we use
$\parallel \cdot
\parallel_{p}$ to represent its norm. Without confusion, we also use $0$ to denote the
zero point of $l^{p}$. Let $\Lambda=\{\lambda_{n}\}_{n=1}^{\infty}$
be a bounded sequence of nonzero complex numbers and let $\pi_{n}$
be the projection from $l^{p}$ to the $n^{th}$ coordinate, i.e.,
$\pi_{n}(x)=x_{n}$ for $n\geq1$. The weighted backward shift
operator $\Lambda B_{p}$ is defined on $l^{p}$ by
$$\pi_{n}\circ \Lambda B_{p}(x)=\lambda_{n}x_{n+1}, \ \ for \  all \ n\geq1,$$
where $x=(x_{1},x_{2},\ldots)$ and $\Lambda$ is called the weight
sequence. For convenience, we use $\lambda B_{p}$ to denote the
constant-weighted backward shift operator on $l^{p}$, i.e., the
associated weight sequence consists of a nonzero constant
$\lambda\in\mathbb{C}$. An excellent introduction to the theory of
such operators and extensive bibliography can be found in the
comprehensive article by Shields \cite{Shields}. In classical
operator theory, we have

\begin{proposition}\label{ppsimilar}

Suppose $1\leq p<\infty$ and $\lambda,\omega\in\mathbb{C}$. Then
$\lambda B_{p}$ and $\omega B_{p}$ are similar if and only if
$|\lambda|=|\omega|$.

\end{proposition}

For each $|\lambda|>1$, $\lambda B_{p}$ is chaotic by Rolewicz
\cite{Rolewicz} and Grosse-Erdmann \cite{Grosse-e}. On the other
hand, if $|\lambda|\leq1$, then the orbit of each point in $l^{p}$
under $\lambda B_{p}$ approaches to the single point $0$ and hence
$\lambda B_{p}$ is not chaotic for $|\lambda|\leq1$. Therefore,

\begin{proposition}\label{01}

Suppose $1\leq p<\infty$, $|\lambda|\leq1$ and $|\omega|>1$. Then
$\lambda B_{p}$ and $\omega B_{p}$ are not topologically conjugate.

\end{proposition}

One can see that $2B_{p}$ and $4B_{p}$ have almost the same
dynamical properties but are not similar. Are there topologically
conjugate? Similarly, are $\frac{1}{2}B_{p}$ and $\frac{1}{4}B_{p}$
topologically conjugate, and what about $\frac{1}{2}B_{p}$ and
$B_{p}$? In this article, we'll answer these questions and at the
end we'll research in general case, the weight sequence being not
constant sequence, by examples.

\section{Homeomorphisms on $l^{p}$}

In the present section, we'll consider the homeomorphisms on
$l^{p}$. Since $l^{p}$ is neither compact nor locally compact, the
homeomorphisms on $l^{p}$ would be of some strange properties, for
instance, the image of a bounded set under a homeomorphism may be
unbounded \cite{Tseng}. However, we also have the following
property.

\begin{lemma}\label{homeo}

Let $f: l^{p}\rightarrow l^{q}$ be continuous, where $1\leq
p,q<\infty$. Then for any $x\in l^{p}$, there is a neighborhood $U$
of $x$ such that $f(U)$ is bounded.

\end{lemma}

\begin{proof}

It is obvious by the continuity of $f$ and the definition of metric
on $l^{p}$.

\end{proof}

Now let's construct some homeomorphisms on $l^{p}$.

For any $p\geq1$ and any $s>0$, define a map $h^{(s)}_{p}$ on
$l^{p}$ as follows,

for any $x=(x_{1},x_{2},\ldots)\in l^{p}$,

$\pi_{n}\circ h^{(s)}_{p}(x)=0$ if $x_{n}=0$ and
$$\pi_{n}\circ h^{(s)}_{p}(x)=\frac{x_{n}}{|x_{n}|} \cdot \sqrt[p]{(\sum_{k=n}^{\infty}|x_{k}|^{p})^{s}-(\sum_{k=n+1}^{\infty}|x_{k}|^{p})^{s}}$$
if $x_{n}\neq0$. Then one can obtain the following result.

\begin{proposition}\label{phomeo}

For any $1\leq p<\infty$ and $s>0$, the map $h^{(s)}_{p}$ is a
homeomorphism from $l^{p}$ onto itself. In fact, the inverse of
$h^{(s)}_{p}$ is $h^{(\frac{1}{s})}_{p}$. Moreover, for any positive
number $\lambda$, $h^{(s)}_{p}(\lambda x)=\lambda^{s}\cdot
h^{(s)}_{p}(x)$for any $x\in l^{p}$.

\end{proposition}
\begin{proof}

Let $\{x^{(m)}\}_{m=1}^{\infty}$ be a Cauchy sequence in $l^{p}$ and
let $y^{(m)}=h^{(s)}_{p}{(x^{(m)})}$ for each $m\geq1$. Write
$x^{(m)}=(x^{(m)}_{1},x^{(m)}_{2},\ldots)$ and
$y^{(m)}=(y^{(m)}_{1},y^{(m)}_{2},\ldots)$, for each $m\geq1$. By
the construction of $h^{(s)}_{p}$, we have
\begin{equation}\label{guji1}
\sum\limits_{n=k}^{\infty}{|y_{n}^{(m)}|^{p}}=(\sum\limits_{n=k}^{\infty}{|x_{n}^{(m)}|^{p}})^{s},
\ \ for \ any \ m, k\in \mathbb{N}.
\end{equation}

So $h^{(s)}_{p}$ is a map from $l^{p}$ to $l^{p}$ and
${\{h^{(s)}_{p}{(x^{(m)})}\}}_{m=1}^{\infty}$ converges by
coordinates, i.e., $\{y_{n}^{(m)}\}_{m=1}^{\infty}$ is a Cauchy
sequence for each $n\in\mathbb{N}$. For any $\epsilon>0$, there is a
positive integer $N_{0}$ such that
$\sum\limits_{n=N_{0}}^{\infty}{|x^{(m)}_{n}|^{p}}<\epsilon$, for
each $m\geq1$. According to  (\ref{guji1}), we have
$$\sum\limits_{n=N_{0}}^{\infty}{|y^{(m)}_{n}|^{p}}<\epsilon^{s},  \ \ \ for \ each \ m\geq1.$$
Consequently, $h^{(s)}_{p}$ is continuous. Notice that both
$h^{(s)}_{p}$ and $h^{(\frac{1}{s})}_{p}$ are continuous, and
$$h^{(s)}_{p}\circ h^{(\frac{1}{s})}_{p}=h^{(\frac{1}{s})}_{p}\circ
h^{(s)}_{p}=id.$$ Therefore, $h^{(s)}_{p}$ is a homeomorphism from
$l^{p}$ onto itself, whose inverse map is $h^{(\frac{1}{s})}_{p}$.
Moreover, it is obvious that for any $x\in l^{p}$,
$$h^{(s)}_{p}(\lambda x)=\lambda^{s}\cdot h^{(s)}_{p}(x).$$

\end{proof}
As is well-known (one can see \cite{Dijk}), for any $1\leq
p,q<\infty$ there is a natural homeomorphism $g_{pq}$ from $l^{p}$
onto $l^{q}$ defined as follows,

for any $x=(x_{1},x_{2},\ldots)\in l^{p}$,

$\pi_{n}\circ g_{pq}(x)=0$ if $x_{n}=0$ and
$$\pi_{n}\circ g_{pq}(x)=\frac{x_{n}}{|x_{n}|} \cdot {|x_{n}|}^{\frac{p}{q}}$$
if $x_{n}\neq0$. And we have

\begin{proposition}\label{pqhomeo}

The inverse of $g_{pq}$ is $g_{qp}$, and for any positive number
$\lambda$, $g_{pq}(\lambda x)=\lambda^{\frac{p}{q}}\cdot g_{pq}(x)$
for any $x\in l^{p}$.

\end{proposition}

\section{Topologically conjugate classes for $\lambda B_{p}$}

In this section we'll give the complete topologically conjugate
classification for $\lambda B_{p}$, which indicates that $2B_{p}$
and $4B_{p}$, $\frac{1}{2}B_{p}$ and $\frac{1}{4}B_{p}$ are
topologically conjugate respectively, but not $\frac{1}{2}B_{p}$ and
$B_{p}$.

For convenience, we define a function $\chi: \mathbb{R}\rightarrow
\mathbb{R}$ as follows
$$\chi(t)=\left\{\begin{array}{cc}
1, \ \ \ &\mbox{if \ $t>1$} \\
0, \ \ \ &\mbox{if \ $t=1$} \\
-1, \ \ \ &\mbox{if \ $t<1$}
\end{array}\right.
$$

\begin{theorem}\label{pp}

Suppose $1\leq p<\infty$ and $\lambda,\omega\in\mathbb{C}$. Then
$\lambda B_{p}$ and $\omega B_{p}$ are topologically conjugate if
and only if $\chi(|\lambda|)=\chi(|\omega|)$.

\end{theorem}

\begin{proof}

Because of Proposition \ref{ppsimilar}, it suffices to consider the
case that both $\lambda$ and $\omega$ are positive number. If
$\chi(|\lambda|)=\chi(|\omega|)$, then there is a positive number
$s$ such that $|\lambda|^{s}=|\omega|$. Consider the homeomorphism
$h_{p}^{(s)}$ defined in the previous section. For any
$x=(x_{1},x_{2},\ldots)\in l^{p}$, denote $x'=(x_{2},x_{3},\ldots)$.
Then, by the construction of $h_{p}^{(s)}$ and Proposition
\ref{phomeo}, we have
$$(h_{p}^{(s)}\circ \lambda B_{p})(x)=h_{p}^{(s)}(\lambda x')=\lambda^{s}\cdot h_{p}^{(s)}(x')=\omega\cdot h_{p}^{(s)}(x'),$$
$$(\omega B_{p}\circ h_{p}^{(s)})(x)=\omega\cdot h_{p}^{(s)}(x').$$
Therefore $h_{p}^{(s)}\circ \lambda B_{p}=\omega B_{p}\circ
h_{p}^{(s)}$ and hence $\lambda B_{p}$ and $\omega B_{p}$ are
topologically conjugate.

On the converse, by Proposition \ref{ppsimilar} and \ref{01} it
suffices to verify that $\lambda B_{p}$ and $B_{p}$ are not
topologically conjugate if $\lambda$ is a positive number less than
$1$. Now suppose a homeomorphism $f$ conjugates $B_{p}$ to $\lambda
B_{p}$, where $0<\lambda<1$. It follows from Lemma \ref{homeo} that
there exist $\delta>0$ and $M>0$ such that $\parallel
f(y)\parallel_{p}<M$ whenever $\parallel y
\parallel_{p}\leq \delta$. Write
$$y^{(n)}=(0,0,\ldots,0,\underbrace{\delta}_{n^{th}},0,\ldots).$$
Since $\parallel y^{(n)}\parallel_{p}=\delta$ for each
$n\in\mathbb{N}$, we have $\parallel(\lambda
B_{p})^{n-1}(f(y^{(n)})\parallel_{p}\leq\lambda^{n-1} M\rightarrow
0$ as $n\rightarrow\infty$. Consequently,
$$f(y^{(1)})=\lim\limits_{n\rightarrow\infty}f(B_{p}^{n-1}(y^{(n)}))=
\lim\limits_{n\rightarrow\infty}(\lambda
B_{p})^{n-1}(f(y^{(n)}))=0.$$ Similarly, one can see
$f(\frac{y^{(1)}}{2})=0$, which is a contradiction.

\end{proof}

Now we've answer the question in section $1$ and it is some
surprising that $B_{p}$ and $\frac{1}{2}B_{p}$ are not topologically
conjugate, although the orbit of each point in $l^{p}$ under either
of them approaches to the single point $0$. Furthermore, we can
obtain a more general result.

\begin{theorem}\label{pq}

Suppose $1\leq p,q<\infty$ and $\lambda,\omega\in\mathbb{C}$. Then
$\lambda B_{p}$ and $\omega B_{q}$ are topologically conjugate if
and only if $\chi(|\lambda|)=\chi(|\omega|)$.

\end{theorem}

\begin{proof}

According to Theorem \ref{pp} and Proposition \ref{ppsimilar}, it
suffices to prove that $\lambda B_{p}$ and $\omega B_{q}$ are
topologically conjugate if $\chi(\lambda)=\chi(\omega)$ where
$\lambda$ and $\omega$ are positive numbers. Now suppose $\lambda$
and $\omega$ are positive numbers with $\chi(\lambda)=\chi(\omega)$.
Then $\chi(\omega^{\frac{q}{p}})=\chi(\lambda)=\chi(\omega)$.
Consider the homeomorphism $g_{pq}$ defined in the previous section.
For any $x=(x_{1},x_{2},\ldots)\in l^{p}$, denote
$x'=(x_{2},x_{3},\ldots)$. By the construction of $g_{pq}$ and
Proposition \ref{pqhomeo}, we have
$$g_{pq}(\omega^{\frac{q}{p}}B_{p}(x))=g_{pq}(\omega^{\frac{q}{p}}x')
=\omega \cdot g_{pq}(x'),$$
$$\omega B_{q}(g_{pq}(x))=\omega \cdot g_{pq}(x').$$
Therefore $\omega^{\frac{q}{p}} B_{p}$ and $\omega B_{q}$ are
topologically conjugate. In addition, $\omega^{\frac{q}{p}} B_{p}$
and $\lambda B_{p}$ are topologically conjugate by Theorem \ref{pp},
hence $\lambda B_{p}$ and $\omega B_{q}$ are topologically conjugate
since topological conjugacy is an equivalent relation.

\end{proof}

We've got that there are three topological conjugate classes for
$\lambda B_{p}$ in all. At the end of this paper, we'll consider the
general case that the weight sequence is not constant sequence, and
we would see some new topologically conjugate classes different from
the three classes.

First of all, we give or restate in following Proposition
\ref{charact} the characterizations of several topologically
conjugate invariances such as chaoticity, transitivity and strong
mixing for weighted backward shift operators. Denote $\beta(n)$ as
$$\beta(n)=\prod\limits_{i=1}^{n}{\omega(i)}, \ \ \ for \ \  \ n=1,2,\ldots \ \ ,$$
where $\{\omega_{n}\}_{n=1}^{\infty}$ is a weight sequence.

\begin{proposition}\label{charact}
If $T$ is a weighted backward shift operator on $l^{p}$, $1\leq
p<\infty$, with weight sequence $\{\omega_{n}\}_{n=1}^{\infty}$,
then

$(I)$(K. G. Grosse-Erdmann \cite{Grosse-e}) $T$ is chaotic if and
only if $\sum\limits_{n=1}^{\infty}\frac{1}{|\beta(n)|^{p}}<\infty$;

$(II)$(G. Costakis and M. Sambarino \cite{Cost})  $T$ is strongly
mixing if and only if
$\lim\limits_{n\rightarrow\infty}|\beta(n)|=\infty$.

$(III)$(H. N. Salas \cite{Salas}) $T$ is transitive if and only if
$\limsup\limits_{n\rightarrow\infty}|\beta(n)|=\infty$.
\end{proposition}

\begin{example}

Let $T^{(1)}$ be a weighted backward shift operator on $l^{2}$ with
the weight sequence $\{\omega_{n}^{(1)}\}_{n=1}^{\infty}$, where
$$\{\omega_{1}^{(1)}\}=1 \ \ and \ \  \{\omega_{n}^{(1)}\}=\sqrt{\frac{n}{n-1}} \ \
for \ \ n\geq 2.$$ It implies $\beta(n)=\sqrt{n}$, for $ n\geq 1$.
Then we have
$\sum\limits_{n=1}^{\infty}\frac{1}{|\beta(n)|^{2}}=\infty$ and
$|\beta(n)|\rightarrow \infty$ as $n \rightarrow \infty$.
Consequently, the operator $T^{(1)}$ is strongly mixing but not
chaotic. Therefore $T^{(1)}$ is not topologically conjugate to
$\lambda B_{2}$ for every $\lambda\in\mathbb{C}$.

\end{example}

\begin{example}

Let $T^{(2)}$ be a weighted backward shift operator on $l^{2}$ with
the weight sequence $\{\omega_{n}^{(2)}\}_{n=1}^{\infty}$, where
$$(\omega_{1}^{(2)},\omega_{2}^{(2)}, \ldots)=(2,\frac{1}{2},2,2,\frac{1}{2},
\frac{1}{2},\underbrace{2,2,2}_{3},\underbrace{\frac{1}{2},\frac{1}{2},\frac{1}{2}}_{3},
\ldots,\underbrace{2,2,\ldots,2}_{k},
\underbrace{\frac{1}{2},\frac{1}{2},\ldots
,\frac{1}{2},\frac{1}{2}}_{k},\ldots).$$ It is easy to see that
$T^{(2)}$ is transitive but not strongly mixing and hence $T^{(2)}$
is not topologically conjugate to either $\lambda B_{2}$ or
$T^{(1)}$.

\end{example}

\begin{example}

Let $T^{(3)}$ be a weighted backward shift operator on $l^{2}$ with
the weight sequence $\{\omega_{n}^{(3)}\}_{n=1}^{\infty}$, where
$$(\omega_{1}^{(3)},\omega_{2}^{(3)}, \ldots)=(\frac{1}{2},2,\frac{1}{2},
\frac{1}{2},2,2,\underbrace{\frac{1}{2},\frac{1}{2},\frac{1}{2}}_{3},\underbrace{2,2,2}_{3},
\ldots,\underbrace{\frac{1}{2},\frac{1}{2},\ldots
,\frac{1}{2},\frac{1}{2}}_{k},\underbrace{2,2,\ldots,2}_{k},
\ldots).$$ It is easy to see that $T^{(3)}$ is not transitive. Now
consider the point
$$x=(0,1,0,
0,0,\frac{1}{2},\underbrace{0,0,0}_{3},\underbrace{0,0,\frac{1}{2^{2}}}_{3},
\ldots,\underbrace{0,0,\ldots
,0,0}_{k},\underbrace{0,0,\ldots,0,\frac{1}{2^{k-1}}}_{k},
\ldots).$$ We have $\parallel(T^{(3)})^{n}(x)\parallel_{2}\geq1$ for
each $n\geq1$. Consequently, $T^{(3)}$ is not topologically
conjugate to $\lambda B_{2}$ for $|\lambda|\leq1$. Therefore,
$T^{(3)}$ is not topologically conjugate to $\lambda B_{2}$ or
$T^{(1)}$, or $T^{(2)}$.

\end{example}

\end{document}